\documentclass[a4paper]{amsart}
\usepackage{amssymb,latexsym}
\usepackage[utf8x]{inputenc}
\usepackage{amsmath,amsthm}
\usepackage{amsfonts,mathrsfs}

\usepackage{enumerate,units}
\usepackage[all]{xy}
\usepackage{graphicx}

\usepackage{xcolor,url}
\usepackage{hyperref}
\hypersetup{
    colorlinks,
    citecolor=blue,
    filecolor=blue,
    linkcolor=blue,
    urlcolor=blue
}
\usepackage{cleveref}

\usepackage{tikz-cd}

\theoremstyle{plain}
\newtheorem{theorem}{Theorem}[section]
\newtheorem{corollary}[theorem]{Corollary}
\newtheorem{lemma}[theorem]{Lemma}
\newtheorem{proposition}[theorem]{Proposition}

\newtheorem{fact}[theorem]{Fact}

\newtheorem*{theorem*}{Theorem}

\theoremstyle{definition}
\newtheorem{definition}[theorem]{Definition}
\newtheorem{example}[theorem]{Example}

\theoremstyle{remark}
\newtheorem{remark}[theorem]{Remark}

\numberwithin{equation}{section}
\newcommand{\forkindep}[1][]{%
  \mathrel{
    \mathop{
      \vcenter{
        \hbox{\oalign{\noalign{\kern-.3ex}\hfil$\vert$\hfil\cr
              \noalign{\kern-.7ex}
              $\smile$\cr\noalign{\kern-.3ex}}}
      }
    }\displaylimits_{#1}
  }
}

\newcommand\R{\mathbb{R}}		
\newcommand\C{\mathbb{C}}
\newcommand\Q{\mathbb{Q}}
\newcommand\N{\mathbb{N}}

\newcommand\U{\mathcal{U}}
\newcommand\F{\mathbb{F}}

\newcommand\OO{\mathcal{O}}
\newcommand\MM{\mathcal{M}}

\newcommand\M{\mathscr{M}}
\newcommand\p{\mathfrak{p}}
\newcommand\q{\mathfrak{q}}
\newcommand\rr{\mathfrak{r}}
\newcommand\Spec{\mathrm{Spec}}
\newcommand\m{\mathfrak{m}}
  \newcommand{\set}[1]{\left\{ {#1} \right\}}
\newcommand{\vect}[1]{\langle {#1} \rangle}
\newcommand{\abs}[1]{\lvert {#1} \rvert}
\newcommand{\sv}[1]{\sqrt{ \vect{ {#1} } }}

\newcommand{\Frac}{\mathrm{Frac}}

\def\dprk{\textrm{dp-rk}}
 \def\char{\textrm{char}}

\begin{document}
\begin{abstract}
It is shown that every dp-minimal integral domain $R$ is a local ring and for every non-maximal prime ideal $\p$ of $R$, the localization $R_\p$ is a valuation ring and $\p R_\p=\p$. Furthermore, a dp-minimal integral domain is a valuation ring if and only if its residue field is infinite or its residue field is finite and its maximal ideal is principal.   
\end{abstract}

\title{Dp-Minimal Integral Domains}

\author[C. d\textquoteright Elb\'ee]{Christian d\textquoteright Elb\'ee$^\dagger$}
\thanks{$^\dagger$ Supported by ISF grant No. 1254/18}
\address{$^\dagger$Einstein Institute of Mathematics\\
	The Hebrew University of Jerusalem\\
	Givat Ram 9190401, Jerusalem\\
	Israel} 
	\email{christian.delbee@mail.huji.ac.il} 
	\urladdr{http://choum.net/\textasciitilde chris/page\textunderscore perso/}

\author[Y. Halevi]{Yatir Halevi$^*$}
\thanks{$^*$Partially supported by  ISF grant No. 181/16 and the Kreitman foundation fellowship.}
\address{$^*$Department of mathematics\\
	Ben Gurion University of the Negev\\
	Be'er Sehva\\
	Israel}
\email{yatirbe@post.bgu.ac.il}
\urladdr{http://ma.huji.ac.il/\textasciitilde yatirh/}

\maketitle

\section{Introduction}
Model theory has a long and fruitful history of classifying algebraic structures under some model theoretic constraints. For example, every $\omega$-stable infinite field is algebraically closed, see \cite{Mac71}. This result was generalized by Cherlin-Shelah to superstable division rings in \cite{ChSh}. Recently, stable division rings of finite dp-rank were also shown to be algebraically closed by Palac\'in and the second author \cite{HaPa}. The end goal of this line of results is the long standing conjecture that every stable field is separably closed. Although there are no example of stable infinite division rings which are not fields, a positive answer to this conjecture would imply that stable division rings are fields \cite[Remark 5.4]{milliet}.

Classifying fields failing to have the independence property (i.e., NIP fields) has had some very encouraging recent results. Starting with the result by Johnson that any dp-minimal infinite field is either algebraically closed, real closed or admits a non-trivial definable henselian valuation \cite{johnson-dp-minimal}, and very recently a generalization of this result by Johnson to fields of finite dp-rank \cite{johnson-dp-finite}, which encompasses the result in \cite{HaPa}. As a generalization of the stable fields conjecture, it is conjectured that every infinite NIP field is either separably closed, real closed or admits a definable henselian valuation.

As for rings, every ($\lambda$-)stable commutative ring with identity can be decomposed in a unique way as a product of $R_1\times \dots\times R_k$ of local rings $(R_i,\M_i)$ such that $\M_i$ is nilpotent and $R_i/\M_i$ is a ($\lambda$-)stable field \cite{ChRe}. There are many other results along this alley, but the moment one assumes that the ring is a stable integral domain we can easily conclude, by DCC on principal ideals, that the ring is a division ring which brings us back to the realm of the first paragraph.

Remaining on the subject of rings but relaxing the model theoretic assumptions from stable to NIP\footnote{There are other results with different model theoretic assumptions, e.g. every O-minimal domain is a division ring and in fact it is RCF, ACF or the quaternions \cite{OtPePi, PeSt}.}, Milliet has shown recently that every NIP division ring of positive characteristic has finite dimension over its center and that there are non-commutative NIP division rings of every characteristic \cite{milliet}. Hempel-Palac\'in have shown that any division ring of burden $n$ has dimension at most $n$ over its center, hence dp-minimal division rings are fields \cite{HePa}.  Dobrowolski-Wagner have shown that every $\omega$-categorical ring of finite dp-rank is virtually null \cite{DoWa}. On the other hand, NIP integral domains that are not valuation rings have not received much attention. The main focus of this paper is dp-minimal integral domains and specifically their connection to valuation rings.

We show that the prime spectrum of an inp-minimal integral domain is linearly ordered by inclusion, and hence a local ring (Corollary \ref{C:prime-ideals-inp-minimal}). A dp-minimal domain is much closer to a valuation ring: every non-maximal prime ideal is the maximal ideal of a valuation overring. The main results are summarized in the following.

\begin{theorem*}
Let $R$ be a dp-minimal integral domain with maximal ideal $\M$.
\begin{enumerate}
    \item $R_\p$ is a valuation ring for every non-maximal prime ideal $\p$;
    \item $R$ is a divided ring, in the sense that $\p R_\p=\p$ for every prime ideal $\p$;
    \item $R$ is a valuation ring if and only if one of the following holds
    \begin{enumerate}
        \item $R/\M$ is infinite;
        \item  $R/\M$ is finite and $\M$ is a principal ideal.
    \end{enumerate} 
    \item if $K=\Frac(R)$ then for every externally definable valuation subring $\OO$ of $K$, either $\OO\subseteq R$ or $R\subseteq \OO$.
\end{enumerate}
\end{theorem*}

\textit{(1)} is proved in Proposition \ref{P:local-is-val} after we show that $R_\p$ with a predicate for $R$ is still dp-minimal for a large enough family of prime ideals $\p$ (Proposition \ref{P:localization-preserves-dp}). \textit{(2)} is proved in Theorem \ref{T:dp-is-divided}. \textit{(3)} and some other connected results are shown in Theorem \ref{T:equi} and \textit{(4)} is Corollary \ref{C:comparaison_valuation_ring_dp-min_dom}. 

Regarding divided rings, they were first studied by Akiba, who used the term ``AV-domains" \cite{Aki67}. Later, Dobbs studied them when searching for an internal characterization for going-down domains (domains $R$ for which $R\subseteq S$ is going down for every overring $S$). Dobbs proved that if $R$ a root-closed domain then $R$ is going-down if and only if $R$ is divided, see \cite{Dob76A} for more information. 

\subsubsection*{Acknowledgments}
We would like to thank Will Johnson for the idea behind Proposition \ref{P:localization-preserves-dp}, Eran Alouf for Lemma \ref{L:eran}(1), Philip Dittman and Itay Kaplan for useful discussions. Finally, we would like to thank Pierre Simon for suggesting this project.

\subsection{Notation and Preliminaries}
A ring will always mean a commutative ring with identity and a domain is an integral domain. For any $a\in R$, the ideal generated by $a$ will be denoted by $aR$, or $\vect{a}$ if it is unambiguous. Let $\char(R)$  denote the characteristic of the ring $R$. We denote by $\Frac(R)$ the fraction field of a domain $R$ and an overring of $R$ is a ring $R\subseteq S\subseteq \Frac(R)$.
Model theoretic notation and notions are standard, see for example \cite{TZ} and \cite{Sim2015}.

\section{General Results}
We start with some general results. The following was already observed by Pierre Simon.

\begin{proposition}\label{P:NIP-rings}
  Let $R$ be an NIP ring and $(\p_i)_{i<\omega}$ an infinite family of prime ideals. Then there is $i_0<\omega$ such that $\p_{i_0}\subseteq \bigcup_{j\neq i_0} \p_j$. Equivalently, there is no infinite antichain of prime ideals and in particular, there is only a finite number of maximal ideals.
\end{proposition}

\begin{proof}
  Assume otherwise and let $a_i\in \p_i\setminus \bigcup_{j\neq i} \p_j$, and for each finite $I\subseteq \omega$, set $b_I = \prod_{i\in I} a_i$. Let $\phi(x,y)$ be the formula $y\in \vect{x}$, i.e. $\exists z (y = zx)$. Then for every $i<\omega$ and finite set $I\subseteq \omega$, $\phi(a_i,b_I)$ holds if and only if $i\in I$. Indeed, if $i\in I$ then it is clear that $b_I\in \vect{a_i}$. Since the ideals $\p_i$ are prime, $b_I\in \bigcap_{i\in I}\p_i \setminus (\bigcup_{j\notin I} \p_j)$ so if $i\notin I$ then $b_I\notin \vect{a_i}$.
\end{proof}
\begin{remark}
The proof of the previous proposition shows that in any commutative ring, the maximal length of an antichain of prime ideals is bounded by the VC-dimension of the formula $y\in \vect{x}$.
\end{remark}
For domains of finite burden, the burden also bounds the maximal length of an antichain of prime ideals:

\begin{proposition}
  Let $R$ be an integral domain of burden $n\in\mathbb{N}$ and let $\p_1,\dots,\p_{n+1}$ be prime ideals of $R$. Then there exists $1\leq i_0\leq n+1$ such that $\p_{i_0}\subseteq \bigcup_{j\neq i_0} \p_j$.
\end{proposition}
\begin{proof}
  Assume not, then for each $1\leq i\leq n+1$ there exists $a_i\in \p_i\setminus \bigcup_{j\neq i} \p_j$. Also, since the ideals are prime, $a_i^k\in \p_i\setminus \bigcup_{j\neq i} \p_j$ for all $k\geq 1$. For each $1\leq i\leq n+1$, and $k\geq 1$, let $X^k_i$ be the set $\vect{a_i^k }\setminus \vect{a_i^{k+1}}$. The latter is nonempty: assume that $a_i^k\in \vect{a_i^{k+1}}$, then for some $b\in R$, $a_i^k = a_i^{k+1} b$. Since $R$ is an integral domain, it follows that $a_i$ is a unit in $R$, which contradicts that $\p_i$ is an ideal. 
  
We now conclude that $\{x\in X_i^k\}_{1\leq i\leq n+1, k\geq 1}$ is an inp-pattern of length $n+1$. Let $k_1,\dots,k_{n+1}\geq 1$. We claim that $a_1^{k_1}\cdot \ldots\cdot  a_{n+1}^{k_{n+1}}\in X_1^{k_1}\cap \dots \cap X_n^{k_n}$. Indeed, if, without loss of generality, $a_1^{k_1}\cdot \ldots\cdot  a_{n+1}^{k_{n+1}}=a_1^{k_1+1}b$, for some $b\in R$, then $a_2^{k_2}\cdot \ldots\cdot  a_{n+1}^{k_{n+1}}\in\vect{ a_1} \subseteq \p_1$. Consequently, $a_j\in \p_1$ for some $j\neq 1$, contradicting the choice of the $a_i$.
To complete the argument, note that the rows are $2$-inconsistent: as before, since $R$ is an integral domain if $1\leq i\leq n+1$ and $s\neq t\geq 1$ then $X_i^s\cap X_i^t=\emptyset$.
\end{proof}

%\begin{remark}[If $R$ is not an integral domain]
 % An idempotent $e$ in a ring $R$ is an element such that $e\neq 0$, $e\neq 1$ and $e^2 = e$.  
%  A ring $R$ is without idempotents if and only if for all nonzero $a,b\in R$, for all $k\in \N$, $a^k = a^{k+1}b \implies 1 = ab$. 
%Let $e$ be an idempotent in $R$, (then in particular $R$ is not an integral domain since $e(1-e) = 0$), then there is a definable injection from $\vect{e}\times \vect{1-e}$ to $R$ given by $(x,y)\mapsto x+y$. Indeed assume for some $(x,y),(x'y')\in \langle e\rangle\times \langle 1-e\rangle$ we have $x+y = x'+y'$ then $x-x' = y'-y\in \vect{e}\cap \vect{1-e} = \set{0}$ ($re = r'(1-e)\implies re^2 = re = r'e(1-e) = 0$). In particular if $R$ is inp-minimal, either $\vect{e}$ or $\vect{1-e}$ is finite.
%\end{remark}

%\begin{remark}
 % A \emph{prime element} $p$ of a ring $R$ is an element such that if $p\mid xy$ then $p\mid x$ or %$p\mid y$. The ideal spanned by a prime element is prime, hence in an integral domain of burden $n$ there are (up to associate) at most $n$ prime elements. It follows that if there exists a prime element in an inp-minimal integral domain, it is unique up to associate.
%\end{remark}

\begin{corollary}\label{C:prime-ideals-inp-minimal}
In an inp-minimal domain the prime spectrum is linearly ordered by inclusion. In particular, all the proper radical ideals are prime and there exists $N\in \mathbb{N}$ such that for all $a,b\in R$ either $b^N\in\langle a \rangle$ or $a^N\in \langle b\rangle$.
\end{corollary}
\begin{proof}
  The fact that the prime ideals are linearly ordered follows directly from the proposition. The fact that all radical ideals are prime follows from the fact that the radical of an ideal is the intersection of all prime ideals containing it. Let $a,b\in R$ then either $\sqrt{\langle{ a\rangle}}\subseteq \sqrt{\langle{ b\rangle}}$ or $\sqrt{\langle{ b\rangle}}\subseteq \sqrt{\langle{ a\rangle}}$, so either $a^n\in \langle b\rangle$ or $b^n\in\langle a\rangle$ for some $n\in \N$. To find a uniform $N$ we use compactness.
\end{proof}

\begin{remark}
A domain with a linearly ordered prime spectrum is also called in the literature a \emph{local treed domain}~\cite[Chapter 7]{Chap00}~\cite{Bad95}.
\end{remark}

\begin{corollary}\label{C:inp-minimal-implies local}
  Let $R$ be an integral domain of burden $n$, then $R$ has at most $n$ maximal ideals. If $R$ is an inp-minimal integral domain which is not a field, then $R$ has exactly one maximal ideal, i.e. $R$ is a local ring.
\end{corollary}

Fields have received more attention in model theory than domains. It would be beneficial to be able to localize without compromising some of the properties of the theory. For that we will require the following lemma.

\begin{lemma}\label{L:tweeking formula}
Let $R$ be an integral domain and $S$ a multiplicatively closed subset of $R$.
Then for every formula $\varphi(\bar x,c)$ in $(S^{-1}R;R,S,+,\cdot,0,1)$, with $c\in S^{-1}R$ there exists a formula $\varphi^*(\bar x,c^*)$ in $(R;S,+,\cdot,0,1)$, with $c^*\in R$ such that
\[\{\bar a\in R: (R;S)\models \varphi^*(\bar a,c^*)\}=\{\bar a\in S^{-1}R:(S^{-1}R;R,S)\models \varphi(\bar a,c)\wedge R(\bar a)\}.\]
\end{lemma}
\begin{proof}
We first prove the following: for every such formula $\varphi(x_1,\dots,x_n,c)$ we can find a formula $\varphi^\dagger(x_1,y_1,\dots,x_n,y_n,c^*)$, with bounded quantifiers (i.e. quantifiers of the form $\exists u\in R$ or $\exists u\in S$) and $c^*\in R$ such that
\[\{(a_1/b_1,\dots, a_n/b_n):(S^{-1}R;R,S)\models \varphi^\dagger(a_1,b_1,\dots,a_n,b_n,c^*), a_i\in R, b_i\in S\}=\]\[\{\bar a\in S^{-1}R:(S^{-1}R;R,S)\models \varphi(\bar a,c)\}.\]

We prove it by induction on the complexity of $\varphi$. Assume $\varphi(x,c)$ is atomic. If it is of the form $S(P(\bar x))$ (or $R(P(\bar x))$) for some polynomial $P(\bar x)$ then we consider the formula $(\exists x_{n+1})(x_{n+1}=P(\bar x)\wedge S(x_{n+1}))$, (or $(\exists x_{n+1})(x_{n+1}=P(\bar x)\wedge R(x_{n+1}))$ and use the induction hypothesis.  As a result, we may assume it is of the form  $R(x_1)$ or $S(x_1)$ or $P(\bar x)=0$ for some polynomial $P$ over $\mathbb{Z}[c]$. If it is of the form $R(x_1)$, the corresponding formula is $(\exists z\in R)(zy_1=x_1)$ and likewise for $S(x_1)$. Assume $\varphi(x,c)$ is of the form \[P(\bar x)=\sum_{i_1,\dots,i_n}c_{i_1,\dots,i_n}x_1^{i_1}\dots x_n^{i_n}=0.\] Then, consider the formula obtained by replacing each $x_i$ by $x_i/y_i$ and clear denominators (from the coefficients as well), i.e. multiplying the whole equation by the product of all the $y_i$'s and the denominators of the $c_{i_1,\dots,i_n}$.

Conjunctions are straightforward and so are negations. If $\varphi(\bar x,c)$ is of the form $(\exists x_{n+1})\psi(x_1,\dots,x_n,x_{n+1},c)$ then the corresponding formula is \[(\exists x_{n+1}\in R, y_{n+1}\in S)\psi^\dagger(x_1,y_1,\dots,x_{n+1},y_{n+1},c^*).\]

Finally we define $\varphi^*(\bar x,c^*)$ to be $\varphi^\dagger(x_1,1,\dots,x_n,1,c^*)$ (after replacing all instances of $R(u)$ by $u$, for any variable $u$).
\end{proof}

\begin{proposition}\label{P:localization-preserves-dp}
Let $R$ be an integral domain and $S$ a multiplicatively closed subset of $R$. 
\begin{enumerate}
    \item If $S$ is definable then the burden of $R$ is equal to then burden of $(S^{-1}R,R)$. In particular if $R$ is NTP$_2$ then so is $(S^{-1}R,R)$.
    \item If $S$ is externally definable in $R$ and $R$ is NIP then $(S^{-1}R,R)$ is NIP and as a result, by $(1)$, $\dprk(R)<\kappa \iff \dprk(S^{-1}R,R)<\kappa$, for any cardinal $\kappa$.  
\end{enumerate}
\end{proposition}

\begin{proof}
\textit{(1)} Since $S$ is definable there is no harm in adding it as a predicate. Since $R$ is definable in $(S^{-1}R,R)$ one direction is obvious. For the other direction, let $\mathcal{N}$ be a monster model and assume that \[\{\varphi_i(x,y_i),(a_{i,j})_{j<\omega},k_i\}_{i<\kappa},\] where $a_{i,j}\in \mathcal{N}$, is an inp-pattern of depth $\kappa$ in $(S^{-1}R;R,S)$, i.e. 
    \begin{itemize}
        \item for all $i<\kappa$, the i-th row is $k_i$ inconsistent
        \item for every $f:\kappa\to \omega$, $\{\varphi_i(x,a_{i,f(i)}\}_{i<\kappa}$ is consistent, witnessed by $b_{f}$.
    \end{itemize}
Since, in $(S^{-1}R;R,S)$, for any finite set $A\subseteq S^{-1}R$ there exists $0\neq c\in R$ with $c\cdot A\subseteq R$, we may find a nonzero element $c\in R(\mathcal{N})$ such that $c\cdot a_{i,j}\in R(\mathcal{N})$ for all $i<\kappa$ and $j<\omega$ and $c\cdot b_f\in R(\mathcal{N})$ for all $f:\kappa\to \omega$. 

Now, for every $i<\kappa$ let $\varphi_i^*(x,y,c^{*_i})$ be the formula corresponding to $\varphi_i(x\cdot c^{-1},y\cdot c^{-1})$ as supplied by Lemma \ref{L:tweeking formula}. It is now routine to check that by Lemma \ref{L:tweeking formula}
\[\{\varphi_i^*(x,y,c^{*_i}),(ca_{i,j})_{j<\omega},k_i)\}_{i<\kappa}\] is an inp-pattern of depth $\kappa$ as witnessed by $(cb_f)_{f:\kappa\to\omega}$.\\
\textit{(2)} Since $S$ is externally definable, by \cite[Corollary 3.24]{Sim2015} there is no harm in adding $S$ to the language. The result follows since $(S^{-1}R,S)$ is interpretable in $(R,S)$. 

% Assume that $(S^{-1}R;R,S)$ is not NIP, thus if $\mathcal{N}$ is a monster model, there exist a $\varphi(x;y)$ a partitioned formula and a finite subset $A\subseteq \mathcal{N}$ and $\{b_I: I\subseteq A\}$, such that 
% \[\varphi(a,b_I)\iff a\in I.\]

% Let $c\in R(\mathcal{N})$ be a non zero element for which $c\cdot A\subseteq R(\mathcal{N})$ and $cb_I\in R(\mathcal{N})$ for all $I\subseteq A$ and let $\varphi^*(x,y,c^*)$ be the formula corresponding to $\varphi(c^{-1}\cdot x;c^{-1}\cdot y)$ as supplied by Lemma \ref{L:tweeking formula}. 
% Consequently, for every $cI\subseteq cA$
% \[ca\in cI\iff a\in I\iff \varphi(a,b_I)\iff\] \[\varphi(c^{-1}ca,c^{-1}cb_I)\iff \varphi^*(ca,cb_I,c^*),\]
% so $\varphi^*(x,y,c^*)$ shatters $cA$, contradiction.
\end{proof}

\begin{corollary}\label{C:dp-min,pNIP-henselian}
Let $R$ be valuation ring. If either $R$ is dp-minimal or $\char(R)=p$ and $R$ is  NIP then $R$ is henselian.
\end{corollary}
\begin{proof}
By \cite[Proposition 4.5]{JSW} (or \cite[Corollary 9.4.16]{johnson}) every dp-minimal valued field is henselian and by \cite[Theorem 2.8]{johnson-positive} every NIP valued field of positive characteristic is henselian. As a result the statement will follow by applying the Proposition with $S=R\setminus \{0\}$.
\end{proof}

To end the section, we note the following, which is just a rephrasing of the  corresponding statement for dp-minimal fields.
\begin{lemma}\label{L:eliminates-infty}
  Let $R$ be an inp-minimal ring.
  \begin{enumerate}
      \item For any two definable sets $X$ and $Y$, if $X$ and $Y$ are infinite then for all $a\in R\setminus\set{0}$, $a(Y-Y)\cap (X-X) \neq \set{0}$. In particular, if $R$ is an integral domain then $R$ eliminates $\exists^\infty$.
      \item $R$ is a field if and only if there exists an infinite definable set $X\subseteq R$ such that $X-X\subseteq R^\times \cup \{0\}$. In particular $R$ admits a definable infinite subring which is a field if and only if $R$ is itself a field.
  \end{enumerate}
  \end{lemma}
\begin{proof}
  \textit{(1)} Otherwise, the function $(x,y)\mapsto x+ay$ would be a definable injection from $X\times Y$ into $R$, contradicting that $R$ is inp-minimal. 
  
  We show that for integral domains, a definable set $X$ is infinite if and only if for all $a\in R$, $a(X-X)\cap (X-X) \neq \set{0}$. The “only if” direction is the previous proof. For the “if”,  let $K=\Frac(R)$. If $R$ is finite, then the result holds trivially. Assume that $R$ is infinite, $X$ is finite, but that for all $a\in R\setminus \{0\}$, $a(X-X)\cap (X-X) \neq \set{0}$, then $R$ is included in $\set{a/b\mid a\in X-X, b\in (X-X)\setminus \set{0}}\subseteq K$ which is a finite set.
  
   \textit{(2)} If $R$ is a field we may take $X=R\setminus \set{0}$. For the other direction, assume that such an infinite definable set $X$ exists. Let $a\in R\setminus \set{0}$ and consider the function $X\times X\to R$ mapping $(x,y)\mapsto ax+y$. Since $R$ is inp-minimal, the map cannot be injective. Hence there exists $u\in X-X\subseteq R^\times$ with $a=u$ so $a$ is invertible.
\end{proof}

\section{Algebraic Properties of Dp-Minimal Domains}
Let $R$ be a dp-minimal domain. In Corollary \ref{C:prime-ideals-inp-minimal} we have seen that there exists $N\in \N$ such that for all $a,b\in R$ either $a^N\in \vect{b}$ or $b^N\in\vect{a}$. A connected notion is that of a divided domain.

\begin{definition}
(\cite{Aki67},\cite{Dob76A}) A prime ideal $\p$ of a ring $R$ is called \emph{divided} if for all $a\in R$ we have $\p \subseteq \vect{a}$ or $\vect{a}\subseteq \p$. Equivalently, a prime ideal $\p$ is divided if and only if $\p = \p R_\p$. A ring $R$ is \emph{divided} if every prime ideal of $R$ is divided.
\end{definition} 

\begin{remark}\label{R:divided}
It is not hard to see that a ring $R$ is divided if and only if for any $a,b\in R$ either $a\in\vect{b}$ or $b^n\in\vect{a}$ for some $n\in \mathbb{N}$. If $\p$ is a divided prime ideal then it is also comparable to every ideal.
\end{remark}

In this section we prove that every dp-minimal ring is divided. There is no hope to assume in general that all prime ideals are (externally) definable. Nevertheless, as the following will show, all prime ideal are an intersection of externally definable prime ideals.

\begin{lemma}\label{L:prime}
Let $R$ be an inp-minimal domain with maximal ideal $\M$.
\begin{enumerate}
    \item For each $a\in R$ there exists a unique ideal $P_a$ such that $P_a$ is maximal with the property $P_a\cap \set{a^n\mid n\in \N} = \emptyset$. $P_a$ is prime and externally definable.

    \item If $R$ is dp-minimal then $(R_{P_a},R)$ is dp-minimal in the language of rings with a predicate for $R$.
    \item For any prime ideal $\p$, $\p=\bigcap_{a\in R\setminus \p}P_a$.
\end{enumerate}
\end{lemma}
\begin{proof}
\textit{(1)} Let $a\in R$. By Zorn's lemma we may find a maximal ideal $P$ satisfying that $P\cap \set{a^n\mid n\in \N} = \emptyset$. Assume that $P$ is not a prime ideal, so there exists $x,y\in R$ such that $xy\in P$ and $x\notin P$, $y\notin P$. By maximality, $P+\vect{x}$ and $P+\vect{y}$ intersect $\set{a^n\mid n\in \N}$, i.e. there exists $c_1,c_2 \in P$ and $r_1,r_2\in R$ such that $c_1+r_1x$ and $c_2+r_2y$ are in $\set{a^n\mid n\in \N}$. It follows that the product $(c_1+r_1x)(c_2+r_2y)$ is in $P\cap \set{a^n\mid n\in \N}$ a contradiction. Uniqueness follows from Corollary~\ref{C:prime-ideals-inp-minimal}. 
    
We show that $R\setminus P$ is externally definable. Consider the uniformly definable family of sets $\varphi(x,a^n):=a^n\in \vect{x}$. The result will follow from \cite[Lemma 3.4]{definable-v-topologies} once we establish that \[\bigcup_{n\in \mathbb{N}}\varphi(R,a^n)= R\setminus P.\]
Let $b\in R\setminus P$, if we would have $\vect{b}\cap\{a^n:n\in \mathbb{N}\}=\emptyset,$ then $\vect{b}$ would be contained in some maximal ideal satisfying this property, by uniqueness $\vect{b}\subseteq P$, contradiction. For the other direction, if $a^n\in\vect{b}$, then we cannot have that $b\in P$.

\textit{(2)} By Proposition \ref{P:localization-preserves-dp}(2) and (1) the structure $(R_{P_a},R)$ is dp-minimal (in the ring language with a predicate for $R$).

\textit{(3)} Let $a\notin \p$. By the definition of $P_a$, $\p\subseteq P_a$. For the other direction, let $b\in\bigcap_{a\notin \p}P_a$. If $b\notin \p$ then $b\in P_b$, contradicting the definition of $P_b$.
\end{proof}

\begin{remark}
Note that for any $a\in R$, $R_{P_a}=S^{-1}R$, where $S=\set{a^n\mid n\in \mathbb{N}}$. The ideals of the form $P_a$ are the so-called Goldman ideals in the literature (see for instance \cite{Pic76}).%\cite{Gue70},\cite{Kap74}
\end{remark}

\begin{fact}\label{F:almost-contain,inp}\cite[Proposition 4.31]{Sim2015}
Let $G$ be a inp-minimal group and $H,N$ definable subgroups. Then either $\abs{H/H\cap N}<\infty$ or $\abs{N/H\cap N}<\infty$.
\end{fact}

\begin{lemma}\label{L:eran}
  Let $R$ be an inp-minimal integral domain with maximal ideal $\M$. If $R$ contains an infinite set $F$ such that $F-F\subseteq R^{\times}\cup\set{0}$ then $R$ is a valuation ring. In particular, 
  \begin{enumerate}
      \item if $R/\M$ is infinite then $R$ is a valuation ring;
    \item if $R$ has an infinite subring which is a field then $R$ is a valuation ring.     
  \end{enumerate}

\end{lemma}
  \begin{proof}
    Assume that such a set $F$ exists, and let $(f_i)_{i<\omega}\subseteq F$ be such that $f_i\neq f_j$ for all $i\neq j$. Let $a,b\in R$ be nonzero elements. Then by Fact \ref{F:almost-contain,inp}, either $\vect{a}/(\vect{a}\cap \vect{b})$ is finite or $\vect{b}/(\vect{a}\cap \vect{b})$ is finite. Without loss of generality, assume that $\vect{a}/(\vect{a}\cap \vect{b})$ is finite. As $(f_ia)_i\subseteq \vect{a}$, it follows that there exists $i\neq j$ such that $(f_ia-f_ja)\in \vect{b}$. As $F-F\subseteq R^{\times}\cup\set{0}$, we conclude that $a\in \vect{b}$.
  \end{proof}

\begin{remark}\label{R:prime-inft-index}
In a local ring, every non-maximal radical ideal has infinite index in the maximal ideal $\M$ (as additive groups). Indeed, if $\rr\subsetneq \M$ is a radical ideal then for any $b\in \M\setminus \rr$ and $n\neq m\in \mathbb{N}$, $b^n$ and $b^m$ are in different classes modulo $\rr$.
\end{remark}

As a result we get the following, which is almost the definition of a Pr\"ufer domain.

\begin{proposition}\label{P:local-is-val}
Let $R$ be a dp-minimal domain with maximal ideal $\M$ and $\p$ a non-maximal prime ideal. Then $R_\p$ is a (possibly trivial) henselian valuation ring.
\end{proposition}
\begin{proof}
Since $\p$ is non-maximal, by Lemma \ref{L:prime}(2) and since the prime ideal are linearly ordered by Corollary \ref{C:prime-ideals-inp-minimal} there exists some $a\in \M\setminus \p$ with $\p\subseteq P_a\subsetneq \M$. This implies that $R_{P_a}\subseteq R_\p$ and as a result it is enough to show that $R_{P_a}$ is a valuation ring.  

We note that since $\abs{R/\p}\leq \abs{R_{P_a}/P_aR_{P_a}}$, the maximal ideal $P_aR_{P_a}$ of $R_{P_a}$ has infinite index. Note that $(R_{P_a},R)$ is dp-minimal by Lemma \ref{L:prime}(2).  By Lemma \ref{L:eran}, $R_{P_a}$ (and thus also $R_\p$) is valuation ring. It is henselian by Corollary \ref{C:dp-min,pNIP-henselian}.
\end{proof}

\begin{theorem}\label{T:dp-is-divided}
Every dp-minimal domain is divided. In particular, there exists $N\in\mathbb{N}$ such that for all $a,b$ in the maximal ideal, either $a\in\vect{b}$ or $b^N\in \vect{a}$.
\end{theorem}
\begin{proof}
Let $R$ be a dp-minimal ring with maximal ideal $\M$ and let $\p$ be a prime ideal. Since $R$ is a local ring, the maximal ideal is divided and so we may assume that $\p\subsetneq \M$. Since $R\cap \p R_\p=\p$, it will be enough to show that $\p R_\p\subseteq R$. By Lemma \ref{L:prime}(3), there exists some $a\in \M\setminus \p$ such that $\p R_\p\subseteq P_aR_{P_a}$ and hence it is enough to show that $P_a$ is divided. Set $P:=P_a$. By Lemma \ref{L:prime}(2), $(R_{P},R)$ is dp-minimal. 

As $PR_{P}$ and $R$ are two definable (additive) subgroups in this structure, by Fact \ref{F:almost-contain,inp}, either $\abs{R/P}<\infty$ or $\abs{PR_{P}/P}<\infty$. Since $P$ necessarily has infinite index in $R$ it must be the latter. Let $y_1,\dots,y_n$ be representatives for the different cosets of $P$ in $PR_{P}$. In particular $\{1,y_1,\dots,y_n\}$ generate $R+PR_{P}$ as an $R$-module, i.e. it is a finitely generated $R$-module.

Either by direct computation or by dp-minimality, it is not hard to see that $R+PR_P$ is a local ring and its maximal ideal is $\M+PR_P$. Consider the dp-minimal ring $R+PR_P$ as an $R$-module. Since $\M(R+PR_P)=\M+PR_P$ and $(R+PR_P)/(\M+PR_P)\cong R/\M$, by Nakayama's lemma the generator $1+(\M+PR_P)$ of $(R+PR_P)/(\M+PR_P)$ lifts to a generator of $R+PR_P$ as an $R$-module, i.e. $R=R+PR_P$, as needed.

The ``in particular" follows by compactness and Remark \ref{R:divided}.
\end{proof}

\begin{remark}
Overrings of the form $R+PR_P$ are called CPI-extensions of $R$~\cite{boisensheldon}. We learned the Nakayama trick for CPI-extensions from~\cite[Lemma 2.3]{Dob78}.
\end{remark}

\begin{corollary}\label{C:dp-min_val_divided_prime}
Let $R$ be a dp-minimal domain. Then $R$ is a (henselian) valuation ring if and only if there exists a non-maximal prime ideal $\p$ such that $R/\p$ is a valuation ring. 

In particular, if the maximal ideal of $R$ is principal and $R$ has a finite residue field then $R$ is a valuation ring.
\end{corollary}
\begin{proof}
If $R$ is a valuation ring and $\p$ is any non-maximal ideal then $R/\p$ is also a valuation ring.

For the other direction, assume that $\p$ is a non-maximal prime ideal and that $R/\p$ is a valuation ring. By Lemma~\ref{P:local-is-val}, $R_\p$ is a valuation ring. By definition, $R+\p R_\p$ is the composition of the valuation ring $R_\p$ and the valuation ring $R/\p\cong (R+\p R_\p)/\p R_\p$ of the field $R_\p/\p R_\p$ (see also \cite[page 724]{boisensheldon}) and hence a valuation ring as well. Since $R$ is divided, $R+\p R_\p=R$ and we conclude that $R$ is a valuation ring.

Let $\M$ be the maximal ideal of $R$, assume that $\M=\vect{\pi}$ is principal and that $R/\M$ is finite. Let $P_{\pi}$ be the maximal ideal that does not intersect $\{\pi,\pi^2,\dots\}$. By definition, for all $0\neq a\in R/P_\pi$, $\pi^n+P_\pi\in aR/P_\pi$ for some $n$. Since $R/\M$ is finite, it follows that $aR/P_\pi$ is of finite index in $R/P_\pi$. Consequently, $R/P_\pi$ is a Noetherian local domain with principal maximal ideal, hence a valuation ring. 
\end{proof}

Let $R$ be any integral domain, $K$ its fraction field and let $I$ be an ideal. We denote
\[I:I=\set{x\in K\mid xI\subseteq I} \text{ and }\]
\[I^{-1}=\set{x\in K\mid xI\subseteq R}.\]

\begin{corollary}
Let $R$ be a dp-minimal domain and $\p$ a non-maximal prime ideal then
\[\p:\p=R_\p.\]

If $\M$ is the maximal ideal and $R$ is not a valuation ring then
\[\M^{-1}=\M:\M.\] 
\end{corollary}
\begin{proof}
The first assertion is a consequence of Theorem~\ref{T:dp-is-divided}, Proposition~\ref{P:local-is-val} and \cite[Lemma 6]{Oka84}.

Let $\M$ be the maximal ideal of $R$ and assume that $R$ is not a valuation ring. By Corollary \ref{C:dp-min_val_divided_prime}, $\M$ is not principal.  Assume there exists $x\in K$ with $x\M\cap (R\setminus \M)\neq \emptyset$ so $x^{-1}R=\M$. This implies that $\M$ is a principal ideal of $R$, contradicting our assumption. 
\end{proof}

For a (type-)definable group $G$ in an NIP structure we denote by $G^{00}$ the smallest type-definable subgroup of $G$ of bounded index, for more information see \cite[Section 8.1.3]{Sim2015}.

\begin{fact}\cite[The proof of Lemma 2.6]{johnson-positive}
Let $R$ be an NIP integral domain and $I$ a definable ideal. Then $I^{00}$ is also an ideal. 
\end{fact}

\begin{proposition}
Let $R$ be a dp-minimal domain with maximal ideal $\M$ and fraction field $K$. Then

\begin{enumerate}
    \item $\M^{00}:\M^{00}$ is a valuation overring of $R$;
    \item $\sqrt{\M^{00}} = \set{a\in \M\mid \M/\vect{a}\text{ is infinite}}$;
    \item For every $a\in R$ with $\M/\vect{a}$ is finite, $\sqrt{\M^{00}} = P_a$.
    \end{enumerate} 
    In particular, exactly one of the following holds:
    \begin{itemize}
        \item $\M/\vect{a}$ is infinite for all $a\in \M$ and $\sqrt{\M^{00}}=\M$; 
        \item $\sqrt{\M^{00}}=P_a$ for any $a\in \M$ with $\M/\vect{a}$ finite.
    \end{itemize}

    %\item $\M=\M^{00}$ if and only if $\M^{00}$ is prime and $\M/\vect{a}$ is infinite for all $a\in \M$.
%\end{enumerate}
\end{proposition}
\begin{remark}
It follows that $\M=\M^{00}$ if and only if $\M^{00}$ is prime and $\M/\vect{a}$ is infinite for all $a\in \M$.
\end{remark}
\begin{proof}
\textit{(1)} It is routine to show that $\M^{00}:\M^{00}$ is a ring, so it is sufficient to prove that it is a valuation ring, i.e. for all $a\in K$, $a\M^{00}\subseteq \M^{00}$ or $a^{-1}\M^{00}\subseteq \M^{00}$. Let $a\in K$. By \cite[Proposition 3.12]{chain-conditions}, $a\M^{00}/(\M^{00}\cap a\M^{00})$ is finite or $\M^{00}/(\M^{00}\cap a\M^{00})$. By multiplying by $a^{-1}$ we may assume without loss of generality that it is the latter. Since $\M^{00}$ has no finite index type-definable subgroups, $\M^{00}=\M^{00}\cap a\M^{00}$, i.e. $a^{-1}\M^{00}\subseteq \M^{00}$.

\textit{(2)} We show that $a\in \M\setminus \sqrt{\M^{00}}$ if and only if $R/\vect{a}$ is finite. Let $a\notin \sqrt{\M^{00}}$. By \cite[Proposition 3.12]{chain-conditions}, $\vect{a}/(\M^{00}\cap \vect{a})$ is finite or $\M^{00}/(\M^{00}\cap \vect{a})$ is finite. If it was the former, we would have that $a^n\in \M^{00}$ for some $n$, i.e. $a\in\sqrt{\M^{00}}$, a contradiction. Consequently, $\M^{00}/(\M^{00}\cap \vect{a})$ is finite and by the definition of $\M^{00}$, $\M^{00}\subseteq \vect{a}$ hence $R/\vect{a}$ is finite. 

Conversely, if $R/\vect{a}$ is finite, then as \[\abs{R/a\M^{00}} = \abs{R/\vect{a}}\abs{\vect{a}/a\M^{00}}= \abs{R/\vect{a}}\abs{R/\M^{00}},\] it follows that $\M^{00} \subseteq a\M^{00}$. As $\M^{00}$ is an ideal we have $\M^{00} = a\M^{00}$. Assume that $a^n\in \M^{00}$ with $n$ is minimal. Thus $a^n=am$ for some $m\in \M^{00}$, but $R$ is a domain so $a^{n-1}\in \M^{00}$, contradiction.

\textit{(3)} Assume that $\M/\vect{a}$ is finite. As $R$ is divided, $P_a\subseteq \vect{a}$. In fact, it follows from the definition that $P_a = \bigcap_n \vect{a^n}$, hence $P_a$ is of bounded index in $\M$ (since each $\vect{a^n}$ is of finite index in $\M$). As a result, $\M^{00}\subseteq P_a$ and hence $\sqrt{\M^{00}}\subseteq P_a$. For the other inclusion, if $b\in P_a$ then $\M/\vect{b}$ is infinite by Remark \ref{R:prime-inft-index} and so, by \textit{(2)}, $b\in \sqrt{\M^{00}}$. 
%\textit{(4)} If $\M = \M^{00}$ it is clear that $\M^{00}$ is prime. Also $\M^{00} = \sqrt{\M^{00}}$, hence $R/\vect{a}$ is infinite for all $a\in \M$ by \textit{(2)}. Conversely, there exists $a\in \M\setminus \M^{00}$. Since $M^{00}$ is a prime ideal by Remark \ref{R:divided}, necessarily $\M^{00}\subseteq \vect{a}$, but $R/\vect{a}$ is unbounded, contradicting the definition of $\M^{00}$.
\end{proof}

We end with the following.

\begin{definition}[\cite{HH78}]
An ideal $\p$ of a domain $R$ is \emph{strongly prime} if for all $x,y\in \Frac(R)$, $xy\in \p$ implies $x\in \p$ or $y\in \p$. A domain is a \emph{pseudo-valuation domain} if every prime ideal is strongly prime.
\end{definition}

Every valuation ring is a pseudo-valuation domain, every pseudo-valuation domain is divided, and a local ring $(R,\M)$ is a pseudo-valuation domain if and only if $\M$ is strongly prime \cite[Theorem 1.4]{HH78}.

\begin{proposition}\label{P:strongly_prime}
Let $R$ be a dp-minimal domain, then every non-maximal prime ideal is strongly prime. 
\end{proposition}
\begin{proof}
By Proposition \ref{P:local-is-val}, if $\p$ is a non-maximal prime ideal of $R$, $R_\p$ is a valuation ring, so in particular it is a pseudo-valuation domain and so $\p R_\p$ is strongly prime. As $R$ is divided, by Theorem \ref{T:dp-is-divided}, $\p = \p R_\p$, so $\p$ is strongly prime.
\end{proof}

\begin{remark}
Dp-minimal domains are not far from being pseudo-valuation domains. However, it is not hard to check that $\F_p+\F_p t +\set{x\mid v(x)\geq 2}$ is a dp-minimal subring of the valued field $(\F_p^{alg}((t^\Q)),v)$, where $v$ is the natural valuation, which is not a pseudo-valuation domain. See Example \ref{E:equip}, for an explanation why the valued field is dp-minimal.
Proposition~\ref{P:nipp} gives a sufficient condition for a dp-minimal domain to be a pseudo-valuation domain.
\end{remark}

\section{Valuation Rings}\label{S:val-rings}
It is well known (and details will be given below) that every dp-minimal valued field of equicharacteristic has an infinite residue field  and that every dp-minimal valued field of mixed characteristic with finite residue field is of finite ramification. We will show below that these are the only obstructions for a dp-minimal ring to be a valuation ring.

Let $R$ be an integral domain with fraction field $K$. $R$ is \emph{root-closed} if for all $q\in K$ such that $q^n\in R$ for some $n\geq 1$ then $q\in R$. It is standard that valuation rings are integrally closed and hence root-closed.

  \begin{theorem}\label{T:equi}
    Let $R$ be an infinite dp-minimal integral domain. 
    \begin{enumerate}
        \item Assume that $R$ is of equicharacteristic $p\geq 0$. The following are equivalent.
         \begin{enumerate}
      \item $R$ is a henselian valuation ring;
      \item $R$ is integrally closed;
      \item $R$ is root-closed;
      \item $R$ has an infinite residue field;
      \item $R$ has an infinite subring which is a field (necessarily $\Q$ or $\F_p^{alg}$);
    \end{enumerate}
    \item Assume that $R$ is of mixed characteristic $(0,p)$ with maximal ideal $\M$. The following are equivalent.
    \begin{enumerate}
    \item[(i)] $R$ is a henselian valuation ring;
    \item[(ii)] Either $R$ has an infinite residue field or $\M$ is principal (and the residue field is finite).
    \end{enumerate}
    \end{enumerate}
          \end{theorem}
  \begin{proof}
\textit{(a)} implies \textit{(b)} is standard and \textit{(b)} implies \textit{(c)} is clear. 

\textit{(c)} implies \textit{(d)} is clear in equicharacteristic $0$. Assume that $R$ is of equicharacteristic $p>0$. By Proposition \ref{P:localization-preserves-dp}, the fraction field $K = \Frac(R)$ is dp-minimal and since it is infinite (as $R$ is infinite), by \cite[Corollary 4.5]{KSW11}, $\F_p^{alg}\subseteq K$. As every element in $\F_p^{alg}$ is an $n$-th root of an element of $\F_p\subseteq R$ for some $n$, it follows that $\F_p^{alg}\subseteq R$. By Lemma~\ref{L:eran}, $R$ is a valuation ring. It follows from \cite[Proposition 5.3]{KSW11} that the residue field is infinite. 

\textit{(d)} implies \textit{(a)} is Lemma~\ref{L:eran}.

\textit{(a)} implies \textit{(e)}. Every local ring of equicharacteristic $0$ contains $\Q$. In equi-characteristic $p>0$, this is contained in the proof of \textit{(c)} implies \textit{(d)}.

\textit{(e)} implies \textit{(a)} is Lemma~\ref{L:eran}.

\textit{(i)} implies \textit{(ii)}. By Proposition \ref{P:localization-preserves-dp}, $(K,R)$ is dp-minimal, where $K=\Frac(R)$. Let $v$ be the valuation $R$ induces on $K$.  If the residue field is finite, by \cite[Theorem 4.3.1]{johnson}, $[0,v(p)]$ is finite. In particular, $\M$ is principal.

\textit{(ii)} implies \textit{(i)}. If $R$ has an infinite residue field then $R$ is a valuation ring by Lemma \ref{L:eran}. If $R$ has a finite residue field and $\M$ is principal then $R$ is a valuation ring by Corollary \ref{C:dp-min_val_divided_prime}.\end{proof}
  
    % \begin{remark}
    % \begin{enumerate}
    %     \item If $R$ is a dp-minimal valuation ring of equi-characteristic $p>0$, then as $K =\Frac(R)$ is Artin-Schreier closed (\cite[Theorem 4.3]{KSW11}) and $R$ is integrally closed, $R$ is Artin-Schreier closed in the algebraic closure of $K$. Lemma~\ref{L:infinitefield} implies that the converse also holds, i.e. if $R$ is Artin-Schreier closed in its fraction field, then $R$ is a valuation ring.
    % \item 
    % \end{enumerate}
    %       \end{remark}
 
 In equicharacteristic $0$, the condition $(d)$ of Theorem~\ref{T:equi} is vacuously true, hence every dp-minimal integral domain of equicharacteristic $0$ is a valuation ring. However there are dp-minimal domains of equicharacteristic $p>0$ which are not valuation rings. 
 
 \begin{example}\label{E:equip}
Consider $K=\F_p^{alg}((t^{\mathbb{Q}}))$, the Hahn series over $\F_p^{alg}$ with value group $\mathbb{Q}$, together with the natural valuation $v$. It is a dp-minimal valued field by e.g. \cite[Theorem 9.8.1]{johnson}. Consider $\F_p+\set{x\mid v(x)\geq 1}$, it is a definable subring of $K$ and hence dp-minimal. It is of equicharacteristic $p>0$ with finite residue field hence by Theorem~\ref{T:equi} it is not a valuation ring. To see this directly, the ideals $\F_p t+\set{x\in K\mid v(x)\geq 2}$ and $\set{x\in K\mid v(x)>1}$ are incomparable. Note that the same argument gives that any ring of the form $\F_p+I$ is not a valuation ring, for any ideal $I$ of the valuation ring of $\F_p^{alg}((t^{\mathbb{Q}}))$.
\end{example}

% In the previous example, taking the integral closure of the dp-minimal ring recovers the original valuation ring. This is not true in general as the next example will show. Specifically, it will show that the integral closure of a dp-minimal ring might not be dp-minimal or a valuation ring.

% \begin{example}\label{E:equip_notval}
% %We give an equicharacteristic $p$ example but a similar construction would work for mixed characteristic (see the basic construction in Example \ref{E:mixed}). 
% Let $\mathbb{F}^{alg}_p\subsetneq k$ be an algebraically closed field and consider $K=k((t^{\mathbb{Q}}))$.  Now consider, as before, $R:=\mathbb{F}_p +\{x\in K: v(x)\geq 1\}$. Every element of the integral closure $\widetilde{R}$ lies in $\mathbb{F}_p^{alg} + \{x\in K:v(x)\geq 1\}$, indeed the residue of such an element must line in $\mathbb{F}_p^{alg}$. Similarly, one can see that the residue field of $\widetilde{R}$ is infinite. Assume that $\widetilde{R}=S+\{x\in K:v(x)\geq 1\}$, for some subring $S\subseteq k$. 
% It suffices to observe that $\widetilde{R}$ is not a valuation ring. As before, we may consider the incomparable ideals $S\cdot t+\{x\in K:v(x)\geq 2\}$ and $\{x\in K: v(x)>1\}$.
% \end{example}

\begin{example}\label{E:mixed}
Let $(\mathbb{Q}_p,v)$ be the p-adics numbers for $p\neq 2$, it is dp-minimal by \cite[Theorem 6.13]{dpminpadic}, and let $K:=\mathbb{Q}_p(\sqrt{p})$ be the totally ramified finite extension given by adjoining the square root of $p$, it is also dp-minimal (together with the valuation $v$). 
The ring $R:=\{0,\dots,p-1\}+\{x\in K: v(x)\geq 1\}$ is definable and hence dp-minimal. Note that the maximal ideal is not principal, $\vect{p}$ is not the maximal ideal since it does not contain $p\sqrt{p}$. In fact, the maximal ideal is generated by $p$ and $p\sqrt{p}$. This shows that it is not a valuation ring and that for dp-minimal rings, $R/pR$ finite does not imply a principal maximal ideal.

%We also note for later reference that the integral closure of $R$ is $\mathbb{Z}_p(\sqrt{p})$ and that it is a finite module over $R$.
\end{example}

\begin{remark}
  Let $\OO_p$ be the valuation ring of the valued field $\F_p^{alg}((t^\Gamma))$, for some divisible ordered abelian group $\Gamma$. Let $I_p$ any ideal of $\OO_p$, then $R_p := \F_p +I_p$ is not a valuation ring (see Example~\ref{E:equip}). Let $\U$ be a non-principal ultrafilter on the set of prime numbers. The ultraproduct $\prod_\U \OO_p$ is dp-minimal since it is the valuation ring of an algebraically closed valued field, however $\prod_\U R_p$ is not even inp-minimal. Indeed, it is not a valuation ring (as none of the $R_p$ are), but it has a pseudo-finite --hence infinite-- residue field. If $\prod_\U R_p$ were inp-minimal, this would contradict Lemma \ref{L:eran}.
  \end{remark}

\section{Externally Definable Valuation Rings}
It is known that every sufficiently saturated non-algebraically closed dp-minimal field admits an externally definable valuation, \cite[Theorem 1.5]{johnson-dp-minimal}. In this section, given a dp-minimal domain, we describe the interactions between $R$ and externally definable valuation rings of $\Frac(R)$.

\begin{lemma}\label{L:RoverOfinite}
Let $R$ be a local domain and $\OO$ a valuation ring of $\Frac(R)$. Then $R/(\OO\cap R)$ is finite if and only if $R\subseteq \OO$.
\end{lemma} 
\begin{proof}
It is clear that $R/(\OO\cap R)$ is finite if $R\subseteq \OO$. Let $\M$ be the maximal ideal of $R$ and $\m$ the maximal ideal of $\OO$. Assume, towards a contradiction, that $|R/(\OO\cap R)|=n\geq 2$. There exists some  $a_1,\dots,a_{n-1}\in R\setminus \OO$ (so $a_i^{-1}\in \m)$ such that \[R = (\OO\cap R)\cup \bigcup_{i=1}^{n-1} ((\OO\cap R)+a_{i})\text{   $(\star)$}.\]

If $a_1\in R^\times$, then $a_1^{-1}\in R\cap \OO$ and by multiplying $(\star)$ by $a_1^{-1}$ we get:
\[a_1^{-1}R = R = a_1^{-1}(\OO\cap R)\cup (a_1^{-1}(\OO\cap R) + 1)\cup \bigcup_{i=2}^{n-1}(a_1^{-1}(\OO\cap R) +a_1^{-1}a_i).\]
Since $a_1^{-1}\OO\subseteq \OO$ and $(a_1^{-1}\OO + 1)\subseteq \OO$, this implies that 
\[R= (\OO\cap R)\cup \bigcup_{i=2}^{n-1}((\OO\cap R) +a_1^{-1}a_i).\]
contradicting that the index is $n$,

If $a_1\notin R^{\times}$ then $a_1\in \M$ hence $1+a_1,(1+a_1)^{-1}\in R^{\times}$. Now translating $(\star)$ by $1$, we get 
\[R+1 = R =((\OO\cap R)+1)\cup \bigcup_{i=1}^{n-1}(\OO\cap R)+1+ a_i) \text{   $(\star')$}.\]
As $\OO$ is a valuation ring, $1+a_1\in \OO$ or $(1+a_1)^{-1}\in \OO$. If it is the former then $((\OO\cap R)+1)\cup((\OO\cap R)+1+ a_1)\subseteq \OO\cap R$ and, as before, we get a contradiction to the index being $n$. If $(1+a_1)^{-1}\in \OO$ then we get a contradiction after multiplying $(\star')$ by $(1+a_1)^{-1}$, as before. As a result, $R = \OO\cap R$, i.e. $R\subseteq \OO$.
\end{proof}

\begin{lemma}\label{L:OoverRfinite}
Let $R$ be a local domain and $\OO$ a valuation ring of $\Frac(R)$. If $\OO/(R\cap \OO)$ is finite then $R[\OO]$ (the ring generated by $R$ and $\mathcal{O}$) enjoys the following properties
\begin{enumerate}
    \item it is a valuation ring;
    \item it is the integral closure of $R$;
    \item it is a finite extension of $R$ and in particular, it is definable in the structure $(\Frac(R),R)$.
\end{enumerate}
%\red{(when does this happen: $\OO$ is a valuation ring and is a finite extension of $R$ ? In fields only if $R$ is real closed)}
\end{lemma} 
\begin{proof}
 If $|\OO/(R\cap \OO)|=1$ then $R[\OO]=R$ is a valuation ring and there is nothing to show, so we may assume that the index is greater than $1$.  Let $R^\prime = \OO\cap R$ and let $a_1,\dots,a_n\in \OO$ be such that $\OO = R^\prime\cup\bigcup_{i=1}^n (R^\prime +a_i)$. Then $\OO =R^\prime[a_1,\dots,a_n]$, and $\OO$ is a finite $R^\prime$-module, in particular $R[\OO]=R[a_1,\dots,a_n]$ is a finite $R$-module. As a result, $R[\OO]$ is an integral extension of $R$ and $R[\OO]$ is definable in $(K,R)$. As $R[\OO]$ is a valuation ring (it contains the valuation ring $\OO$) it is integrally closed and hence it is the integral closure of $R$.
\end{proof}

For valuation rings, the following lemma is a straight application of the approximation theorem for incomparable valuation rings. By assuming inp-minimality, it also follows for integrally closed rings.

\begin{lemma}\label{L:subring_of_inp-min_fields}
Let $R_1$ and $R_2$ two integrally closed rings and $K = \Frac(R_1) = \Frac(R_2)$. If $(R_1,R_2)$ is inp-minimal, then $R_1\subseteq R_2$ or $R_2\subseteq R_1$.
\end{lemma}
\begin{proof}
By inp-minimality, $R_1/(R_1\cap R_2)$ or $R_2/(R_1\cap R_2)$ is finite. Assume the former, then $R_1$ is an integral extension of $R_1\cap R_2$. As $R_1\cap R_2$ is again integrally closed, we must have that $R_1 = R_1\cap R_2$, i.e. $R_1\subseteq R_2$.
\end{proof}

\begin{definition}
Let $(R,\M)$ be a local ring and $(R^\prime,\M^\prime)$ a local overring of $(R,\M)$. We say that $R^\prime $ is a  \emph{dominant} extension of $R$  if $\M = R\cap \M^\prime$ (equivalently $\M\subseteq \M^\prime$), and \emph{non-dominant}, otherwise.
\end{definition}

\begin{proposition}\label{P:val_overring_localisation}
Let $(R,\M)$ be a dp-minimal domain $K = \Frac(R)$. 
\begin{enumerate}
    \item If $\OO$ is a local non-dominant overring of $R$ (not necessarily externally definable in $(K,R)$), then $\OO$ is a valuation ring and there exists a non-maximal prime ideal $\p$ of $R$ such that $\OO = R_\p$.
    \item If $R\subseteq \OO$ is a local dominant valuation overring such that $(R,\OO)$ is dp-minimal (e.g. if $\OO$ is externally definable), then for any non-maximal prime ideal $\p$ of $R$, $\OO\subseteq R_\p$. Furthermore, $\Spec(R)\setminus \set{\M}$ is an initial segment of $\Spec(\OO)\setminus \set{\m}$.
\end{enumerate}
\end{proposition}
\begin{proof}
\textit{(1)} Since $R\subseteq \OO$ is non-dominant, $\p=\m\cap R$ is a non-maximal prime ideal of $R$, where $\m$ is the maximal ideal of $\OO$. By Proposition \ref{P:local-is-val}, $R_\p$ is valuation overring of $R$. As $R_\p\subseteq \OO$ (if $a\in R$, $s\in R\setminus \p$, then $s\notin \m$ hence $a/s\in \OO$), we have that $\OO$ is also a valuation ring. Since $\OO$ is a localization of $R_\p$, $\OO$ is a localization of $R$.

\textit{(2)} Let $\p$ be a non-maximal prime ideal of $R$. By Proposition \ref{P:local-is-val}, $R_\p$ is a (non-dominant) valuation overring of $R$. Since $(R,\OO)$ is dp-minimal, by Proposition \ref{P:localization-preserves-dp}, $(R_p,\OO)$ is dp-minimal as well so by Lemma \ref{L:subring_of_inp-min_fields}, $\OO\subseteq R_\p$ or $R_\p\subseteq \OO$. As $R\subseteq \OO$ is dominant, necessarily $\OO\subseteq R_\p$. We need now to show that $\p$ is a prime ideal of $\OO$. To show that it is an ideal, we note that $\p\subseteq \p\OO\subseteq \p R_\p=\p$, where the last equality is by Theorem
~\ref{T:dp-is-divided}. The ideal $\p$ is a prime ideal of $\OO$ since it is a prime ideal of $R_\p$. 

To show that it is an initial segment, we show that if $\q$ is a prime ideal of $\OO$ strictly contained in $\M$ then $\q$ is a prime ideal of $R$. Indeed, $\q \subseteq R$ and $\q R \subseteq \q \OO\subseteq \q$ so $\q$ is an ideal of $R$. Further it is prime in $\OO$ so in particular in $R$.
\end{proof}

\begin{corollary}\label{C:comparaison_valuation_ring_dp-min_dom}
Let $R$ be a domain and $\OO$ a valuation ring of $\Frac(R)$ such that $(R,\OO)$ is dp-minimal. Then $R\subseteq \OO$ or $\OO\subseteq R$.
\end{corollary}
\begin{proof}
By dp-minimality, $\OO/(\OO \cap R)$ is finite or $R/(\OO\cap R)$ is finite. If $R/(\OO\cap R)$ is finite, then by Lemma~\ref{L:RoverOfinite}, $R\subseteq \OO$. Thus we may assume that $\OO/(\OO\cap R)$ is finite. If $R/(\OO\cap R)$ is a non-dominant extension then, by Proposition~\ref{P:val_overring_localisation} $R$ is a valuation overring of $\OO\cap R$ and hence by Lemma
~\ref{L:subring_of_inp-min_fields} $R$ and $\OO$ are comparable. 

So we may assume that $R$ is a dominant overring of $\OO\cap R$. We prove that $R\subseteq \OO$. Let $a\in R$. If $a$ is non-invertible in $R$ then $a$ is non-invertible in every dominant valuation overring of $R$. By Lemma~\ref{L:OoverRfinite} $R[\OO]/R$ is an integral extension of local rings and hence a dominant extension, so $a$ is non-invertible in $R[\OO]$. It follows that $a^{-1}\notin \OO$ hence $a\in \OO$. 

Now, if $a$ is invertible in $R$ and $a\notin \OO$, then $a^{-1}\in \OO\cap R$ and $a\notin \OO\cap R$, hence $a^{-1}$ is in the maximal ideal of $\OO\cap R$. Since $R$ is a dominant extension of $\OO\cap R$, $a^{-1}$ is also in the maximal ideal of $R$, contradiction.
\end{proof}

\begin{remark}
In characteristic $p>0$, Corollary~\ref{C:comparaison_valuation_ring_dp-min_dom} has a more direct proof. Let $\OO$ be an externally definable valuation subring of $K = \Frac(R)$.  Lemma~\ref{L:RoverOfinite} takes care of the case  where $R/(\OO\cap R)$ is finite so we may assume that $\OO/(\OO\cap R)$ is finite and set $R^\prime=\OO\cap R$. Since $\OO$ is dp-minimal, by Theorem \ref{T:equi}, $\mathbb{F}_p^{alg}\subseteq \OO$. As a result, $(R^\prime+\mathbb{F}_p^{alg})/R^\prime\cong R^\prime/(\mathbb{F}_p^{alg}\cap R^\prime)$ is finite as well. It is easy to see that every subring of $\mathbb{F}_p^{alg}$ is a field, hence $\mathbb{F}_p^{alg}\cap R^\prime$ is a subfield of $\mathbb{F}_p^{alg}$ of finite codimension and in particular an infinite field. Hence $R^\prime$ contains an infinite field and so $R'$ is a valuation ring by Theorem \ref{T:equi}, and so is $R$. By Lemma
~\ref{L:subring_of_inp-min_fields}, $R$ and $\OO$ are comparable.
\end{remark}

% By Lemma  \ref{C:externallyvaluation_dpminimalfield}, $\OO$ and $R_\p$ are comparable and $\OO$ is dominant, we must have $\OO \subseteq R_\p$. If $\p$ is a non-maximal prime ideal of $R$, we need to show that $\p$ is a prime ideal of $\OO$. As $\p$ is non-maximal, $R_\p$ is a non-dominant valuation overring of $R$ and we have $\OO\subseteq R_\p$. It follows that $\OO \p = \p R_\p = \p$ since $R$ is divided, so $\p$ is an ideal of $\OO$, and the same arguments yields that $\p$ is a prime ideal of $\OO$. If $\p\subsetneq \M$ is a prime ideal of $\OO$, then it is in particular a prime ideal of $R$ which is non-maximal. 

% \textit{(3)} For each $a\in R$, $R_{P_a}$ is an externally definable valuation overring of $R$, so $\check R$ is a valuation overring of $R$. Now we assume that $\M$ has no immediate predecessor. As the maximal ideal of $R_{P_a}$ is $P_a$, it is easy to see that the maximal ideal of $\check R$ is $\bigcup_{a\in \M} P_a$ which equals $\M$ if $\M$ has no immediate predecesssor. Let $(\OO,\m)$ be an externally definable dominant valuation overring of $R$, then by Corollary~\ref{C:externallyvaluation_dpminimalfield} $R_{P_a}$ and $\OO$ are comparable. If $R_{P_a}\subseteq \OO$ then $\M = \m\cap R\subseteq \m\cap R_{P_a}\subseteq P_a$, a contradiction. 
% It follows that $\OO\subseteq R_{P_a}$ for all $a\in \M$, so $\OO\subseteq \check R$. 

\section{The Prime Spectrum of Dp-minimal Domains}\label{S:prime_spectrum}
We end with some non-elementary results concerning the prime spectrum of certain dp-minimal domains.

% In~\cite{EK19}, Echi and Khalfallah give a fine description of the prime spectrum of divided rings. They use those results to exhibit interesting properties of the prime spectrum of convex subrings of $^*\C$ and $^*\R$, the fields of hypercomplex and hyperreals. 

% They prove that in the ring of bounded hypercomplex $^b\C$ and the ring of bounded hyperreals $^b\R$, the prime spectrum  does not have three consecutive elements.

% We extend their result about immediate predecessor ideals to the class of local treed domains, i.e. domains in which the prime spectrum is linearly ordered by inclusion. 

Let $R$ be a domain with $\Spec(R)$ linearly ordered by inclusion (e.g. if $R$ is inp-minimal). For $\p\subsetneq \q  \in \Spec(R)$, we will say that $\p$ is a \emph{predecessor} of $\q$ and that $\q$ is a \emph{successor} of $\p$. For $a\in \M$ we observe the following:
 \begin{itemize}
     \item $\sqrt{\vect{a}}$ is the minimal prime ideal containing $\vect{a}$;
     \item $P_a$ is the maximal prime ideal not containing $a$;
     \item $P_a$ is the immediate predecessor of $\sqrt{\vect{a}}$.
 \end{itemize}
 
Note that if $a\in \M$ then $P_a\neq \M$, if $a\neq 0$ then $\sv a$ is not equal to the zero ideal and that $P_a$ for $a=0$ does not make any sense.

\begin{lemma}\label{L:immediate_predecessor}
Let $R$ be a domain with $\Spec(R)$ linearly ordered by inclusion. For $ \p\subsetneq \q\in Spec(R)$, the following are equivalent:
\begin{enumerate}
    \item $\p$ is the immediate predecessor of $\q$;
    \item $\p = P_a$ and $\q = \sv a$ for any $a\in \q\setminus \p$.
\end{enumerate}
\end{lemma}
\begin{proof}
Assume that $\p$ is an immediate predecessor of $\q$. For $a\in \q\setminus \p$, necessarily $\p=P_a$ and $\q=\sv a$ by definition. The other direction follows by the above discussion.
\end{proof}

\begin{proposition}\label{P:weaklysat}
Let $R$ be a domain with $\Spec(R)$ linearly ordered by inclusion. The following are equivalent.
\begin{enumerate}
\item $\set{P_a\mid a\in \M} \setminus \set{\M}$ is densely ordered by inclusion;
    \item there are no three consecutive elements in $\Spec(R)$;
    \item for all $a,b\in \M$ with $a,b\neq 0$, $P_a\neq \sv b$.
    
\end{enumerate}
\end{proposition}
\begin{proof}
\textit{(1)} implies \textit{(2)} is obvious.

\textit{(2)} implies \textit{(3)}. Assume that $P_a=\sv b$ for some $a,b\in \M$ with $a,b\neq 0$. By the above discussion, $P_b\subsetneq \sv b=P_a\subsetneq \sv a$ are three consecutive prime ideals, contradicting \textit{(2)}.

\textit{(3)} implies \textit{(1)}. Assume there are $P_a\subsetneq P_b\subsetneq \M$, for $a,b\in \M$, necessarily non-zero elements, with $P_a\subsetneq P_b$ consecutive prime ideals. By Lemma \ref{L:immediate_predecessor}, $P_b=\sv c$ and $P_a=P_c$ for some $c\in P_b\setminus P_a$, contradicting the assumption. %Now, assume there exists $\vect{0}\subsetneq P_a\subsetneq \M$ with $0\neq a\in \M$ and $\vect{0}\subsetneq P_a$ consecutive prime ideals. This contradicts Lemma \ref{L:prime}(3).
\end{proof}

%Let $I_a = \bigcap_{n\geq 1} \vect{a^n}$. It is easy to check that a domain $R$ is divided if and only if for all $a\in R$ we have $P_a = I_a$.
\begin{definition}
We say that a domain $R$, with $\Spec(R)$ linearly ordered by inclusion, has property $(\star)$ if it satisfies one of the equivalent conditions of Proposition~\ref{P:weaklysat}.
\end{definition}

An integral domain of finite Krull dimension has property $(\star)$ if and only if it is one-dimensional. It is not hard to see that any $\aleph_0$-saturated domain has property $(\star)$, and actually this is true in higher generality.

\begin{proposition}\label{P:saturated-implies-v-sat}
Let $D$ be a $\kappa$-saturated domain and $R$ a $\bigvee$-definable local subring whose prime ideals are linearly ordered.  Then $R$ has property $(\star)$.
\end{proposition}
\begin{remark}
By $R$ $\bigvee$-definable we mean $R$ is definable by $\bigvee_{i<\lambda}\varphi_i(x,a_i)$ for some $\lambda<\kappa$.
\end{remark}
\begin{proof}
Assume that $R$ is defined by $\bigvee_{i<\lambda} ``x\in R_i"$, where $\lambda<\kappa$. There is no harm is assuming that the $R_i$ are closed under finite unions. Let $\M$ be the maximal ideal and assume towards a contradiction that $P_a=\sv b$, as prime ideals of $R$, for some $a,b\in \M$ with $a,b\neq 0$. Note that $\sv b$ is $\bigvee$-definable:
\[\bigvee_{n\in \mathbb{N}} (x^n\in bR)=\bigvee_{n\in\mathbb{N}}((\exists y\in R) (x^n=yb))=\bigvee_{n\in \mathbb{N}}\bigvee_{i<\lambda}( (\exists y\in R_i) (x^n=yb)).\]
As for $P_a$, it is defined by the type
\[\bigwedge_{n\in\mathbb{N}}(a^n\notin xR)= \bigwedge_{n\in\mathbb{N}}(\neg (\exists y\in R)(a^n=xy))=\bigwedge_{n\in\mathbb{N}}\bigwedge_{i<\lambda}((\forall y\in R_i)(a^n\neq xy)).\]
Since $P_a = \sv b$, by compactness there exist finite subsets $F\subseteq \mathbb{N}$ and $I\subseteq \lambda$ such that 
\[x\in P_a \iff \bigwedge_{n\in F}\bigwedge_{i\in I}((\forall y\in R_i)(a^n\neq xy)).\]
Since the $R_i$ are closed by finite unions, there is some $i<\lambda$ such that 
\[x\in P_a \iff \bigwedge_{n\in F}((\forall y\in R_i)(a^n\neq xy)).\]
Let $n_0=\max\set{n:n\in F}+1$. Since $a$ is non-invertible in $R$, $a^{n_0}y\neq a^n$ for all $y\in R$ (in particular for all $y\in R_i$) and $n\in F$. As a result, $a^{n_0}\in P_a$, contradiction.
\end{proof}

\begin{remark}\label{R:v-sat_VR}
Let $\OO$ be a valuation ring with value group $\Gamma$. It is not hard to see that $\OO$ has property $(\star)$ if and only if the set of archimedean components of $\Gamma$ is densely ordered. Indeed, by using the standard correspondence between prime ideals and convex subgroups of the value group, each archimedean component corresponds to $\sv a/P_a$ for some $0\neq a\in M$. 
% In particular if $\Gamma$ is $\omega$-saturated, then $\OO$ is v-saturated. For $\Gamma$ any $\omega$-saturated ordered abelian group, and any field $K$, the valuation ring of the field of Hahn series $K((t^\Gamma))$ is v-saturated. 
\end{remark}
\begin{remark}\label{R:Echi_Khalfalah}
Let $^{*}\R$  be the hyperreals (resp. $^*\C$ the hypercomplex) and $^b\R$ (resp. $^b\C$) the ring of bounded elements. $^b\R$ and $^b\C$ are $\bigvee$-definable valuation rings in an $\aleph_1$-saturated domain, hence Proposition~\ref{P:saturated-implies-v-sat} applies. In \cite{EK19}, Echi and Khalfallah prove directly that these valuation rings do not have three consecutive prime ideals. The previous proposition is a generalisation of this result.
% This also gives another proof of a result from \cite{EK19} concerning the hyperreals $\R^*$ (and hypercomplex $\C^*$), since the subring of bounded elements $\R^b$ (respectively $\C^b$) is the valuation ring of the associated valuation on $\R^*$ (respectively $\C^*$) which has $\omega$-saturated value group.

% (in both case it is $\R^*$, see for instance \cite{AG14} for more on that subject).
\end{remark}

% Let $\R^*$ be an $\omega$-saturated model of $\RCF$ and $\R^b$ the subring of $\R^*$ consisting of the bounded elements of $\R^*$, i.e. the smallest convex subgroup of $\R^*$ containing $0$. By a \emph{convexe subring} of $\R^*$, we mean a subring of $\R^*$ containing $\R^b$, (see \cite{EK19}), in particular it is a localised of $\R^b$. The result of Echi and Khalfallah follows from Remark~\ref{R:v-sat_VR}. Let $v$ be the valuation on $\R^*$ where positive valuation elements are infinitesimals. Then $\R^b$ is the valuation ring of $(\R^*,v)$ and the value group is $\omega$-saturated (it is order isomorphic to $\R^*$\red{true? ref?}), $\R^b$ is v-saturated (although not $\omega$-saturated), i.e. there are no three consecutive elements in the prime spectrum of $\R^b$. Any convex subrings of $\R^*$ is a localised of $\R^b$ hence is also v-saturated. Note that the same argument yields the results of Echi and Khalfallah for convex subrings of the hypercomplex $\C^*$.

\begin{proposition}\label{P:v-sat -down in dominant}
Let $(R,\M)$ be a local ring and $R\subseteq \OO$ a dominant valuation overring with $(R,\OO)$ dp-minimal. If $\OO$ has property $(\star)$ then so does $R$.
\end{proposition}
\begin{proof}
By Proposition~\ref{P:val_overring_localisation} (2), if $\OO$ does not have three consecutive prime ideals then the same holds for $\Spec(R)$.
\end{proof}

\begin{corollary}
Let $(K,v)$ be a dp-minimal field with a valuation ring $\OO$ having  property $(\star)$ (e.g. if $vK$ is $\aleph_0$-saturated). Let $(R,\M)$ be an externally definable ring in $(K,v)$.
\begin{enumerate}
    \item If $v(\M)>0$ then $R$  has property $(\star)$.
    \item If the residue field of $(K,v)$ is finite then $R$  has property $(\star)$.
\end{enumerate}
\end{corollary}
\begin{proof}
\textit{(1)} By Corollary~\ref{C:comparaison_valuation_ring_dp-min_dom}, $R\subseteq \OO$ or $\OO\subseteq R$. If the former holds, then apply Proposition \ref{P:v-sat -down in dominant}. If the latter holds, then $R$  has property $(\star)$, since $\Spec(R)$ is an initial segment of $\Spec(\OO)$.

\textit{(2)} As in \textit{(1)}, we may assume that $R\subseteq \OO$. By Proposition~\ref{P:v-sat -down in dominant}, it is enough to show that $\OO$ is a dominant extension of $R$. Otherwise, by Proposition \ref{P:val_overring_localisation}, $\OO = R_\p$ for some non-maximal prime ideal $\p$ of $R$. As $R$ is divided by Theorem \ref{T:dp-is-divided},
\[|R/\p|\leq |R_\p/\p|=|R_\p/\p R_\p|<\infty,\]
contradicting Remark \ref{R:prime-inft-index}.
\end{proof}

% \begin{remark}
% Let $\Gamma$ be an $\omega$-saturated ordered abelian group (for instance, a non-principal ultrapower of $\R$, $\Q$ or $\Z$). Then any externally definable subring of $\F^{alg}_p[[t^\Gamma]]$ (if $\Gamma$ is $p$-divisible), $\C[[t^\Gamma]]$ or $\Q_p[[t^\Gamma]]$ whose maximal ideal has positive $t$-valuation is v-saturated. Similarly any externally definable subring of a non-principal ultrapower of $\Z_p$ is v-saturated.
% \end{remark}

% \begin{remark}
% Note that there are saturated dp-minimal valuation rings in which the maximal ideal has an immediate prime predecessor: any elementary extension of $\Z_p$, the ideal $P_p = I_p$ is an immediate prime predecessor of the maximal ideal $\vect{p}$.
% \end{remark}

In a domain with linearly ordered prime spectrum, $\aleph_0$-saturation also gives that $P_a\neq \set{0}$ for all $a$. Divided domains satisfying this property are called \emph{pointwise non-archimedean} \cite{Dob86} \cite{AKP98}.  A somewhat mirror notion is requiring that $\sv{a}\neq \M$ for all $a$, or equivalently, the maximal ideal does not have an immediate prime predecessor.

%Once can think about $v$-saturation as having many prime ideals near the zero ideal. A somewhat mirror notion is requiring that $\M$ does not have an immediate prime predecessor.

\begin{proposition}\label{P:nipp}
Let $R$ be a dp-minimal domain with maximal ideal $\M$. If $\M$ has no immediate prime predecessor then 
\begin{enumerate}
    \item $R$ is a pseudo-valuation domain;
    \item $\M/\vect{a}$ is infinite for all $a\in \M$;
    \item $\M^{00}=\M$;
    \item $\M^{-1}=\bigcap_{a\in \M} R_{P_a}$ is a valuation overring with maximal ideal $\M$;
    \item for every $a,b\in \M$ either $a\in \vect{b}$ or $b^2\in\vect{a}$.
\end{enumerate}
\end{proposition}
\begin{proof}
\textit{(1)} It is sufficient to prove that $\M$ is strongly prime. Assume that $x,y\in \Frac(R)$ with $xy\in \M$. As $\M$ does not have an immediate predecessor, $\sv{xy}\subsetneq \M$ (for otherwise, $P_{xy}$ is an immediate predecessor of $\M$). Since $\sv{xy}$ is a prime ideal, Proposition \ref{P:strongly_prime} gives $x\in \sv{xy}$ or $y\in\sv{xy}$, as needed. 

\textit{(2)} If $\M/\vect{a}$ is finite for some $a\in \M$ then for every $b\in \M$ there exists some $n$ and $m$ such that $b^n-b^m\in\vect{a}$ hence $b^k\in \vect{a}$ for some $k$, i.e. $\M=\sv a$. Contradicting the fact that $\M$ does not have an immediate predecessor. 

\textit{(3)} Assume there exists $a\in \M\setminus \M^{00}$. Since $\M$ does not have an immediate predecessor $\sv a\subsetneq \M$ and as $\sv a$ is a prime ideal and $\M^{00}$ is an ideal, by Remark \ref{R:divided} and Theorem \ref{T:dp-is-divided}, $\sv a$ and $\M^{00}$ are comparable. Since $a\in \sv a$ but $a\notin \M^{00}$, necessarily $\M^{00}\subseteq \vect{a}\subseteq \sv a$, but it follows from $(2)$ that $\M/\vect{a}$ is unbounded, and hence so is $\M/\M^{00}$, contradiction.

\textit{(4)} Consider $\check R=\bigcap_{a\in \MM} R_{P_a}$. Since the prime ideals are linearly ordered, by Proposition \ref{P:local-is-val}, $\check R$ is a valuation overring of $R$. Its maximal ideal is $\bigcup_{a\in \MM}P_aR_{P_a}=\bigcup_{a\in \MM}P_a$, where the last equality holds since $R$ is divided by Theorem \ref{T:dp-is-divided}. Since $\M$ has no prime predecessor, by Lemma \ref{L:prime} we may conclude that $\bigcup_{a\in \MM}P_a=\M$. The result now follows by \cite[Theorem 2.7 and Theorem 2.10]{HH78}.

\textit{(5)} follows from \textit{(4)} and \cite[page 4368]{Bad95}.
% \cite[Proposition 2.6(b)]{ABDS09}, or
\end{proof}

\begin{remark}
If $R$ has an immediate prime predecessor, $\M=\sv a$ for some $a\in \M$. By Theorem \ref{T:dp-is-divided}, there exists an $n\in \mathbb{N}$ such that $b^n \in \vect{a}$ for all $b\in \M$. This fits well with the examples at the end of Section \ref{S:val-rings}. For each of these kind of examples, $R$ is is a pseudo-valuation ring or there exists some ideal $\vect{a}$ and $n\in\mathbb{N}$ such that $b^n\in\vect{a}$ for all $b\in\M$.
\end{remark}
\begin{remark}
A domain $R$ is said to be \emph{fragmented}  if for each non-invertible $a\in R$ there exists a non-invertible $b\in R$ such that $a\in \bigcap_n \vect{b^n}$~\cite{Dob85, CD01}. In a divided domain, $P_a = \bigcap_n \vect{a^n}$, and it is easy to see that a divided domain is fragmented if and only if the maximal ideal has no immediate prime predecessor. Consequently, every fragmented dp-minimal domain is a pseudo-valuation ring. 

If $\OO$ is a dp-minimal domain of positive characteristic having property $(\star)$ then for any $a\in \OO$, $\F_p+P_a$ is a fragmented pseudo-valuation domain which is not a valuation ring.
\end{remark}

\bibliographystyle{alpha}
\bibliography{dpmin-rings}
\end{document}